\documentclass[reqno]{amsart}
\usepackage{amsmath,amsthm,amssymb}
\usepackage{latexsym}
\usepackage{eucal}

\newtheorem{theorem}{Theorem}[section]
\newtheorem{lemma}{Lemma}[section]

\newtheorem{proposition}{Proposition}[section]

\theoremstyle{definition}
\newtheorem{definition}{Definition}[section]

\newtheorem{open problem}{Open Problem}
\newtheorem{remark}{Remark}

\def\cal{\mathcal}
\let\Re=\undefined
\DeclareMathOperator{\Re}{Re}
\let\Im=\undefined
\DeclareMathOperator{\Im}{Im}

\begin{document}
\title[Wave equation with slowly decaying potential\ldots \ldots
 ]{Wave equation with slowly decaying potential: asymptotics of
 solution and wave operators}

\author{Sergey A. Denisov}
\address{
\begin{flushleft}
University of Wisconsin--Madison\\  Mathematics Department\\
480 Lincoln Dr., Madison, WI, 53706, USA\\
  denissov@math.wisc.edu
\end{flushleft}
  }
 \maketitle
\begin{abstract}
In this paper, we consider one-dimensional wave equation with
real-valued square-summable potential. We establish the long-time
asymptotics of solutions by, first, studying the stationary problem
and, secondly, using the spectral representation for the evolution
equation. In particular, we prove that part of the wave travels
ballistically if $q\in L^2(\mathbb{R}^+)$ and this result is sharp.
\end{abstract} \vspace{1cm}

\section{Introduction}

We will consider the Cauchy problem for the wave equation
\begin{eqnarray}
y_{tt}(x,t)=y_{xx}(x,t)-q(x)y(x,t), \quad y(x,0)=\phi(x),\quad
y_t(x,0)=\psi(x),\nonumber \\ y(0,t)=0,\quad t\geq 0, x\in
\mathbb{R}^+ \label{wave}
\end{eqnarray}
where $q(x)$ is time-independent, real-valued, and $q(x)\in
L^2(\mathbb{R}^+)$. In the free case (i.e., when $q=0$),
\[
y(x,t)=\left(\frac{\phi_o(x+t)+\phi_o(x-t)}{2}+\frac{1}{2}
\int_{x-t}^{x+t} \psi_o(s)ds\right)\chi_{x>0}
\]
where $\phi_o$ and $\psi_o$ are odd continuations of $\phi$ and
$\psi$ to $\mathbb{R}$.

 Assuming that $\phi$ and $\psi$ are smooth and compactly
supported, this formula gives the classical solution that propagates
with finite speed.

If $q\neq 0$, the stationary operator for the problem is
one-dimensional Schr\"odinger

\begin{equation}
Hf=-f''+qf,\quad f(0)=0, \quad f\in W^{2,2}(\mathbb{R}^+)
\label{Schrodinger}
\end{equation}
The conservation law is
\begin{equation}
E(t)=\langle y_t,y_t\rangle+\langle Hy,y\rangle ={\rm const}
\label{conservation}
\end{equation}

In this paper, we will study the long-time behavior of $y(x,t)$
which is the main question of the scattering theory for
(\ref{wave}).
 The scattering theory of wave
equation is the classical subject and was extensively studied in the
literature (see \cite{lax} for an excellent source of information).
It is known that the spectrum of stationary operator can be pure
point as long as $q\in L^p(\mathbb{R}^+)$ with $p>2$ \cite{naboko,
simon3} and that will show the sharpness of our result in the
$L^p$--scale for potential. The method is a modification of the one
employed for Dirac operator \cite{den1}. We heavily relied on the
``trace identities" obtained in \cite{ks}. The paper \cite{ks} can
also be a source for the proofs of various well-known results we
mention in the text. In the context of polynomials orthogonal on the
unit circle, similar statements were proved  in \cite{ds1, den2,
simon2}.\bigskip

The organization of the paper is as follows: in the second section,
we introduce the modified Jost function and obtain its
multiplicative representation. In the third section, we prove
asymptotics of the generalized eigenfunctions for the stationary
Schr\"odinger operator. In the fourth one, we apply this result to
prove theorem on asymptotics of solution to wave equation. For that,
we need to additionally assume that potential oscillates in a
certain way. The last section contains the proof of the existence of
modified wave operators in general case. In the Appendix, we
collected some standard results we used in the paper.

We will use the following definition for the (inverse) Fourier
transform of the function $f(x)$
\[
\hat{f}(k)=\int_\mathbb{R} f(x)e^{ik x}dx
\]
\bigskip\bigskip

\section{Schr\"odinger operator with $L^2(\mathbb{R}^+)$
potential: some preliminaries}

\subsection{Modified Jost solutions}

Let $u(x,k)$ be the solution to
\[
-u''+qu=Eu, \quad u(0,k)=0, u'(0,k)=1
\]
In this paper we will always take $E=k^2$. Assume that $q\in
L^2(\mathbb{R}^+)$ which by Weyl's theorem implies $\sigma_{\rm
ess}(H)=[0,\infty)$.
 Denote the spectral measure of $H$ by
$d\rho(E)$. We have the following generalized Plancherel identity
\[
\int_\mathbb{R}|\breve{f}(E)|^2d\rho(E)=\|f\|_2^2
\]
for any $f\in L^2(\mathbb{R}^+)$ where
\[
\breve{f}(E)=\int_0^\infty f(x)u(x,k)dx
\]
and the integral is understood in the $L^2$--sense. The case of
square summable potential is well-understood now, e.g., the
criterion for $q\in L^2(\mathbb{R}^+)$ in terms of $d\rho$ was
obtained in \cite{ks}. Let us write
\[
d\rho(E)=d\rho_s(E)+\mu(E)dE,\quad E>0
\]
where $d\rho_s$ is singular part and $\mu$ is the density of a.c.
part.

 We will need to control the asymptotical behavior of $u(x,k)$
for $x\to+\infty$. The pointwise in $k$ results of that type were
obtained in \cite{christ} for $q\in L^p(\mathbb{R}^+)$ with $p<2$.
We will obtain the asymptotics of $u(x,k)$ for real $k$ but in the
integral sense. To this end, we take $k\in \mathbb{C}^+$ and define
the  modified Jost functions as follows.\bigskip

For arbitrary large $R$, consider the truncation
$q_R(x)=q(x)\chi_{x<R}$. Now that $q_R\in L^1(\mathbb{R}^+)$, we can
define the usual Jost solution $j(x,k,R)$  by requiring
\begin{equation}
-j''+q_Rj=k^2j, \quad j(x,k,R)=e^{ikx}, x>R \label{jost}
\end{equation}
The function $j(x,k,R)$ exists not only for any $k\in \mathbb{C}^+$
but also for real $k\neq 0$. That, for instance, follows from the
integral equations for $j(x,k,R)$. Indeed, rewrite (\ref{jost}) as

\[
J'= \left[
\begin{array}{cc}
0 & 1\\
q_R-k^2 & 0
\end{array}
\right]J, \quad J=\left(
\begin{array}{c}
j\\
j'
\end{array}
\right)
\]
and for
\[
Z=\left[
\begin{array}{cc}
e^{ikx} & e^{-ikx}\\
ike^{ikx} & -ike^{-ikx}
\end{array}
\right]^{-1}J
\]
we have
\begin{equation}
 Z'=\frac{iq_R}{2k}\left[\begin{array}{cc}
-1 & -e^{-2ikx}\\
e^{2ikx} & 1
\end{array}\right]Z, \quad Z=(z_1,z_2)^t, \quad Z(R)=(1,0)^t \label{mono}
\end{equation}
Diagonalizing
\begin{equation}
Z=\left[
\begin{array}{cc}
e^{i\phi} & 0\\
0 & e^{-i\phi}
\end{array}
\right]\Psi, \quad \phi(x,k,R)=(2k)^{-1}\int_x^R q(t)dt
\label{relation-1}
\end{equation}
we finally get
\[
\Psi'=\frac{iq_R}{2k} \left[
\begin{array}{cc}
0 & -e^{-2ikx-2i\phi(x,k,R)}\\
e^{2ikx+2i\phi(x,k,R)} & 0
\end{array}
\right]\Psi, \quad \Psi=(\psi_1,\psi_2)^t
\]

 Due to asymptotics of $j$ at infinity, we have integral equations
 for $\psi_1,\psi_2$:

\begin{eqnarray}
\psi_1(x,k,R)=1+\frac{i}{2k}\int_x^\infty q_R(t)e^{-2ikt-2i\phi(t,k,R)}\psi_2(t,k,R)dt\nonumber\\
\psi_2(x,k,R)=-\frac{i}{2k} \int_x^\infty
q_R(t)e^{2ikt+2i\phi(t,k,R)}\psi_1(t,k,R)dt \label{ie}
\end{eqnarray}

Then, one can easily show that iterations converge for any $k\in
\overline{\mathbb{C}^+}\backslash \{0\}$. Moreover, if $k\in
\mathbb{C}^+$, then this convergence is uniform in $R$ since $q\in
L^2(\mathbb{R}^+)$. The function $j(0,k,R)$ is called the Jost
function for $q_R$, the truncated $q$. We will denote it by
$j(k,R)$. Clearly, we have
\begin{eqnarray*}
j(k,R)=j(0,k,R)=\exp(i\phi(0,k,R))\psi_1(0,k,R)+\exp(-i\phi(0,k,R))\psi_2(0,k,R)=\\
=\exp(i\phi(0,k,R))j_m(k,R)
\end{eqnarray*}
where
\begin{equation}
j_m(k,R)=\psi_1(0,k,R)+\exp(-2i\phi(0,k,R))\psi_2(0,k,R)\label{relation0}
\end{equation}
We will call $j_m(k,R)$ the modified Jost function for truncated
$q$. The function $j(k,R)$ is analytic in $k$ in $\mathbb{C}^+$ and,
in general, might have zeroes on $i\mathbb{R}^+$. If these zeroes
are denoted by $\{ik_j(R)\}$, then $\{-k_j^2(R)\}$ are the negative
eigenvalues of $H_R$.  Since $j(x,k,R)$ is well-defined for real
$k\neq 0$ and the Wronskian $W(j(x,k,R), \overline{j(x,k,R)})=2ik$,
we have
\[
u(x,k)=(2ik)^{-1}(\overline{j(k,R)}j(x,k,R)-j(k,R)\overline{j(x,k,R)}),\quad
x\in [0,R]
\]
In particular, taking $x=R$, we get
\begin{eqnarray}
u(R,k)=(2ik)^{-1}(\overline{j(k,R)}e^{iRk}-j(k,R)e^{-iRk})=\quad\quad\quad\quad\quad\nonumber \\
=(2ik)^{-1}(\overline{j_m(k,R)}e^{iRk-i\phi(0,k,R)}-j_m(k,R)e^{-iRk+i\phi(0,k,R)})
\label{relation}\\\nonumber\\
k^2u(R,k)^2+u'(R,k)^2=|j(k,R)|^{2}=|j_m(k,R)|^{2}\label{modul}
\end{eqnarray}

Later on, we will need a formula for $\partial_R j(k,R)$.
\begin{lemma} For any real $k\neq 0$, we have
\begin{equation}
\partial_R
j(k,R)=\frac{q(R)}{2ik}(-j(k,R)+\overline{j(k,R)}e^{2ikR})\label{derivative}
\end{equation}
\end{lemma}
\begin{proof}
Denote the fundamental matrix for (\ref{mono}) by $Y(x_1,x)$ such
that
\[
\partial_{x}Y=\frac{iq_R(x)}{2k}\left[\begin{array}{cc}
-1 & -e^{-2ikx}\\
e^{2ikx} & 1
\end{array}\right]Y, \quad Y(x_1,x_1)=I
\]
Notice that
\begin{equation}
Y(R,x)= \left(
\begin{array}{cc}
z_1(x,R) & \overline{z}_2(x,R)\\
z_2(x,R) & \overline{z}_1(x,R) \end{array}\right)\label{idez1}
\end{equation}
by (\ref{mono}). Write
\begin{eqnarray*}
z_1(x,R)=1+\int_x^R \frac{iq(t)}{2k}z_1(t,R)dt+\int_x^R
\frac{iq(t)}{2k}e^{-2ikt}z_2(t,R)dt\\
z_2(x,R)=-\int_x^R \frac{iq(t)}{2k}e^{2ikt}z_1(t,R)dt-\int_x^R
\frac{iq(t)}{2k}z_2(t,R)dt
\end{eqnarray*}
for $x<R$. Differentiating in $R$ for fixed $x$,
\begin{eqnarray*}
\partial_Rz_1(x,R)=\frac{iq(R)}{2k}+\int_x^R \frac{iq(t)}{2k}\partial_Rz_1(t,R)dt+\int_x^R
\frac{iq(t)}{2k}e^{-2ikt}\partial_Rz_2(t,R)dt\\
\partial_Rz_2(x,R)=-\frac{iq(R)}{2k}e^{2ikR}-\int_x^R \frac{iq(t)}{2k}e^{2ikt}\partial_Rz_1(t,R)dt-\int_x^R
\frac{iq(t)}{2k}\partial_Rz_2(t,R)dt
\end{eqnarray*}
and this identity is for $L^2$ functions in $R$. Differentiating in
$x$, we have
\[
(\partial_R z_1(x,R),\partial_R
z_2(x,R))^t=\frac{iq(R)}{2k}Y(R,x)(1,-e^{2ikR})^t
\]
Substitute (\ref{idez1}) and
\[
z_1(0,R)=\frac{1}{2ik}(ikj(0,k,R)+j'(0,k,R)),\quad
z_2(0,R)=\frac{1}{2ik} (ikj(0,k,R)-j'(0,k,R))
\]
to get (\ref{derivative}).
\end{proof}
Our goal is to prove asymptotics of $j_m(k,R)$ as $R\to\infty$.
Having that at hand, we can obtain the asymptotics for $u(R,k)$ by
(\ref{relation}). At this point, it is crucial that we work with
analytic extension of $j_m(k,R)$ to the upper half-plane.\bigskip

As we mentioned before, one can prove that $j_m(k,R)\to j_m(k)$
uniformly over the compacts in $\mathbb{C}^+$ by working with
integral equations (\ref{ie}). Instead, one can use the
determinantal representation for $j(k,R)$ which was advocated in
\cite{killip}. It is a rather well-known fact that
\begin{equation}
j(k,R)=\det \left(\frac{H_R-k^2}{H^0-k^2}\right)=\det
(I+R^0(k^2)q_R)\label{relation2}
\end{equation}
where $H^0=-\partial^2$ with Dirichlet condition at zero,
$R^0(z)=(H^0-z)^{-1}$, and $\det$ is the perturbation determinant.
The regularization of perturbation determinant gives
\[
\det (I+R^0(k^2)q_R)=\exp({\rm tr} (R^0(k^2)q_R)){\det}_2
(I+R^0(k^2)q_R)
\]
Since
\begin{equation}
[R^0f](x)=-\int\limits_0^\infty
\frac{e^{ik|x-y|}-e^{ik(x+y)}}{2ik}f(y)dy\label{reso}
\end{equation}
we obtain \begin{eqnarray} \det (I+R_0(k^2)q_R)=\exp\left( -\int_0^R
\frac{1-e^{2ikx}}{2ik}q(x)dx\right){\det}_2 (I+R_0(k^2)q_R)=\\
=\exp\left(-\frac{1}{2ik}  \int_0^R
q(x)dx\right)\exp\left((2ik)^{-1} \hat{q}_R(2k)\right){\det}_2
(I+R_0(k^2)q_R)\nonumber
\end{eqnarray}
where $\hat{q}_R(k)$ is the (inverse) Fourier transform of $q_R$. By
(\ref{relation-1}), (\ref{relation0}), and (\ref{relation2}),
\[
j_m(k,R)=\exp\left((2ik)^{-1} \hat{q}_R(2k)\right){\det}_2
(I+R_0(k^2)q_R)
\]
so\bigskip

\begin{definition}({\bf Modified Jost function})
\[
j_m(k,R)\to j_m(k)=\exp\left((2ik)^{-1} \hat{q}(k)\right){\det}_2
(I+R_0(k^2)q)
\]
\end{definition}\bigskip
and the convergence is uniform over the compacts in $\mathbb{C}^+$
due to $q\in L^2(\mathbb{R}^+)$ and the standard properties of the
regularized determinants.\bigskip

For the Weyl-Titchmarsh function $m_R(z)$ of truncated potential, we
have two different representation. One is through the spectral
measure
\[
m_R(z)=\int\left(\frac{1}{E-z}-\frac{E}{1+E^2}\right)d\rho_R(E)+C_R,
\quad \int_\mathbb{R} \frac{d\rho_R(E)}{1+E^2}dE<C
\]
and the other one is through the so-called Weyl's solution, which in
our case  is a multiple of the Jost solution, so
\[
m_R(z)=\frac{j'(0,k,R)}{j(0,k,R)}, \quad z=k^2
\]
Since $q_R$ is compactly supported, $d\rho_R$ is absolutely
continuous on $\mathbb{R}^+$ and $j(0,k,R)$ is continuous up to the
real line which gives

\[
\pi
\mu_R(E)=\frac{1}{2i}\left(\frac{j'\overline{j}-\overline{j'}j}{|j|^2}\right)=k|j(k,R)|^{-2}
\]
In the next section, we will show that the same factorization
identity holds for $j_m(k)$ and $\mu(E)$.

\subsection{The $a(k)$--coefficient and its modification}

Take the truncated potential $q_R$, continue it to $\mathbb{R}^-$ by
zero, and consider the Schr\"odinger equation on the line. Then, for
the Jost solution
\[
j(x,k,R)=a(k,R)e^{ikx}+b(k,R)e^{-ikx}, \quad x<0
\]

The functions $a$ and $b$ possess many interesting properties that
we list below: \bigskip
\begin{itemize}
\item[1.]
The identity $|a(k,R)|^2=1+|b(k,R)|^2$ is true for any $k\in
\mathbb{R}\backslash\{0\}$. Consequently, $|a(k,R)|\geq 1$ on the
real line. Also, $a(-k,R)=a(k,R)$ so $|a(k,R)|$ is even.
\item[2.]  The function $a(k,R)$ is analytic  in $k\in \mathbb{C}^+$ and is
continuous up to $\mathbb{R}\backslash \{0\}$. It can have zeroes
only on $i\mathbb{R}^+$. If they are denoted by $\{i\xi_j(R)\}$,
then the negative eigenvalues of the operator
\[
-\frac{d^2}{dx^2}+q_R\chi_{x>0}, \quad x\in \mathbb{R}
\]
are exactly $\{-\xi_j^2(R)\}$.
\item[3.] The functions $m$, $j$, $a$, and $b$ are related by
\[
j(k,R)=a(k,R)+b(k,R), \quad
m_R(k^2)=ik\frac{a(k,R)-b(k,R)}{a(k,R)+b(k,R)}
\]
\begin{equation}
a(k,R)=\left(\frac{m_R(k^2)+ik}{2ik}\right)j(k,R) \label{id3}
\end{equation}\bigskip

By (\ref{relation0}) and (\ref{id3}), we have
\[
a(k,R)=\exp\left( \frac{i}{2k}\int_0^\infty
q_R(x)dx\right)a_m(k,R),\quad  a_m(k,R)=
\left(\frac{m_R(k^2)+ik}{2ik}\right)j_m(k,R)
\]
so $a_m$ and $a$ are different only by a singular function that (in
its multiplicative representation) corresponds to a jump at zero
with the mass \[  \sim\int_0^\infty q_R(x)dx\] which can grow in
$R$.\bigskip

\item[4.] We have
\begin{equation}
a_m(k,R)=1-\frac{1}{8ik^3} \int_0^\infty
q_R^2(x)dx+\bar{o}(k^{-3})\label{asimut}
\end{equation}
 as $\Im k\geq 0$ and $|k|\to\infty$ (indeed, that can be easily observed since $a_m(k,R)=\psi_1(0,k,R)$
from~(\ref{ie})).
\end{itemize}\bigskip\bigskip

\begin{definition}{\bf (Modified $a(k)$--coefficient)}
\[
a_m(k,R)\to a_m(k)=\left(\frac{m(k^2)+ik}{2ik}\right)j_m(k),\quad
k\in \mathbb{C}^+
\]
\end{definition}
Similarly to (\ref{asimut}),
\begin{equation}
a_m(k)=1-\frac{1}{8ik^3} \int_0^\infty q^2(x)dx+\bar{o}(k^{-3}),
\quad \Im k\to+\infty\label{as}
\end{equation}
One can easily show that $a_m(k)$ is analytic in $\mathbb{C}^+$. If
its zeroes are denoted by $\{i\xi_j\}$, then the negative
eigenvalues of
\[
H_l=-\frac{d^2}{dx^2}+q\chi_{x>0}, \quad x\in \mathbb{R}
\]
are exactly $\{-\xi_j^2\}$.

\bigskip\bigskip
Next, let us obtain the integral representation for $a_m(k,R)$ and
$a_m(k)$. We have
\[
a(k,R)=B(k,R)\exp\left(\frac{1}{i\pi}\int_\mathbb{R}
\frac{1+tk}{(t-k)(1+t^2)}\ln |a(t,R)|dt\right)
\]
and $B(k,R)$ is the  standard Blaschke product
\[
B(k,R)=\prod_{j} \frac{k-i\xi_j(R)}{k+i\xi_j(R)}
\]
Taking $k=iy$, $y\to +\infty$ we have
\[
\prod_{j}
\frac{y-\xi_j(R)}{y+\xi_j(R)}\exp\left(\frac{1}{\pi}\int_\mathbb{R}\frac{y}{t^2+y^2}
\ln |a(t,R)|dt\right)\sim 1+\frac{1}{2y}\int_0^\infty q_R(x)dx
\]
Since $|a(t,R)|\geq 1$ for real $t$,
\[
\int_\mathbb{R} \ln |a(t,R)|dt=\frac \pi2 \int_0^\infty
q_R(x)dx+2\pi\sum_j \xi_j(R)
\]
and, since $|a(t,R)|$ is even in $t$,
\begin{equation}
a_m(k,R)=B_m(k,R)\exp\left( \frac{1}{\pi i} \int_\mathbb{R} K(t,k)
\ln |a_m(t,R)|dt\right) \label{id5}
\end{equation}
where
\[
K(t,k)=\frac{t^2}{k^2(t-k)}
\]
is a modified Cauchy kernel. For the modified Blaschke product,
\[
B_m(k,R)=\prod_{j}
\left(\frac{k-i\xi_j(R)}{k+i\xi_j(R)}\exp(2ik^{-1}\xi_j(R))\right)
\]

We can rewrite (\ref{id5}) as
\begin{equation}
a_m(k,R)=B_m(k,R)\exp\left( \frac{1}{2\pi i} \int_\mathbb{R} K(t,k)
\ln \left( \frac{|m_R(t^2)+it|^2}{4\pi |t|\mu_R(t^2)} \right)dt
\right)\label{id6}
\end{equation}

By \cite{ks}, formula $(1.33)$,
\begin{equation}
\frac 23\sum_j \xi_j^3 +\frac{1}{\pi} \int_0^\infty t^2 \ln \left(
\frac{|m(t^2)+it|^2}{4\pi |t|\mu(t^2)} \right)dt=\frac 18
\int_0^\infty q^2(x)dx \label{relq1}
\end{equation}
if $q\in L^2(\mathbb{R})$, where
\[
\frac{|m(t^2)+it|^2}{4\pi |t|\mu(t^2)} \geq 1
\]
for a.e. $t$.

Notice that since $|m(t^2)+it|^2\geq t^2$, we have for any $I\subset
\mathbb{R}^+$
\begin{equation}
\int_{I} \ln \mu(t^2)dt>-C_1-C_2\int_0^\infty q^2(x)dx
\label{uniformia}
\end{equation}

\begin{definition}{\bf (Modified Blaschke product).}\hspace{0.5cm}
If $\{-\xi_j^2\}$ are negative eigenvalues of $H_l$, then we define
\[
B_m(k)=\prod_{j}
\left(\frac{k-i\xi_j}{k+i\xi_j}\exp(2ik^{-1}\xi_j)\right)
\]

\end{definition}

\begin{lemma}
For any $k\in \mathbb{C}^+$,
\[
a_m(k)=B_m(k)\exp\left( \frac{1}{2\pi i} \int_\mathbb{R} K(t,k) \ln
\left( \frac{|m(t^2)+it|^2}{4\pi |t|\mu(t^2)} \right)dt \right)
\]
\end{lemma}
\begin{proof}
Let us introduce the function
\[
h(k)= k^2\int_\mathbb{R} K(t,k) \ln \left( \frac{|m(t^2)+it|^2}{4\pi
|t|\mu(t^2)} \right)dt=\int_\mathbb{R} \frac{t^2}{t-k} \ln \left(
\frac{|m(t^2)+it|^2}{4\pi |t|\mu(t^2)} \right)dt
\]
This last integral is the Cauchy-type integral of a non-negative
function from $L^1(\mathbb{R})$. For the truncated potential, we
take
\[
f(k,R)= \int_\mathbb{R} \frac{t^2}{t-k} \ln \left(
\frac{|m_R(t^2)+it|^2}{4\pi |t|\mu_R(t^2)} \right)dt=2\pi i k^2 \ln
\left(a_m(k,R)B_m^{-1}(k,R)\right)
\]
For the imaginary parts of $f(k,R)$ and $h(k)$,
\[
\Im f(x+iy,R)=\int_\mathbb{R}\frac{yt^2}{(t-x)^2+y^2}\ln \left(
\frac{|m_R(t^2)+it|^2}{4\pi |t|\mu_R(t^2)} \right)dt
\]
and
\begin{equation}
\Im h(x+iy)=\int_\mathbb{R}\frac{yt^2}{(t-x)^2+y^2}\ln \left(
\frac{|m(t^2)+it|^2}{4\pi |t|\mu(t^2)} \right)dt\label{relq}
\end{equation}

By repeating the proof of theorem 5.1 (\cite{ks}, page 20), one has
\begin{equation}
\Im h(k)\leq \liminf_{R\to \infty} \Im f(k,R), \quad k\in
\mathbb{C}^+ \label{pros}\bigskip
\end{equation}

$\Bigl($This proof we refer to is based on several standard
observations. We sketch the details here following \cite{ks}. If
\[
r_R(k)=\frac{ik-m_R(k^2)}{ik+m_R(k^2)},\quad
r(k)=\frac{ik-m(k^2)}{ik+m(k^2)}
\]
then $r_R(k)\to r(k)$ as $\Im k>0, \Re k>0$ and
\begin{equation}
T(k)+|r(k)|^2=1, \quad T_R(k)+|r_R(k)|^2=1 \quad k\in \mathbb{R}
\label{enez}
\end{equation}
where
\[
T(k)=\frac{4k\pi \mu(k^2)}{|m(k^2)+ik|^2}, \quad T_R(k)=\frac{4k\pi
\mu_R(k^2)}{|m_R(k^2)+ik|^2}
\]
Since both $r_R(k)$ and $r(k)$ have nice boundary behavior near
$\mathbb{R}^+$, one can prove $r_R(k)\to r(k)$ weakly in $L^p(a,b)$
for any $(a,b)\subset \mathbb{R}^+$ and $p\in [1,\infty)$. Thus, for
any $j=1,2\ldots$
\[
\int_a^b \frac{yt^2}{(t-x)^2+y^2} |r(t)|^{2j} dt\leq \liminf_R
\int_a^b \frac{yt^2}{(t-x)^2+y^2} |r_R(t)|^{2j} dt
\]
Using the Taylor expansion for $\ln (1-|r|^2)$, identities
(\ref{enez}), and summing up the estimates above, we have
\[
\int_a^b \frac{yt^2}{(t-x)^2+y^2} \ln T(t) dt\geq \liminf_R \int_a^b
\frac{yt^2}{(t-x)^2+y^2} \ln T_R(t) dt
\]
for any $a,b$. Taking $a\to 0, b\to\infty$ and handling
$\mathbb{R}^-$ similarly, we get (\ref{pros}).$\Bigr)$
\bigskip\bigskip

 On the other hand,
\begin{equation}
f(k,R)\to f(k)={2\pi i k^2} \ln (a_m(k)B_m^{-1}(k))\label{relw}
\end{equation}
so
\[
0\leq \Im h(k)\leq \Im f(k), \quad k\in \mathbb{C}^+
\]

Computing the asymptotics as $\Im k\to+\infty$, we have
\[
f(k)\sim -\frac{\pi}{4 k}\int_0^\infty
q^2(x)dx+\frac{4\pi}{3k}\sum_j \xi_j^3
\]
 by (\ref{relw}) and (\ref{as}). Similarly,
\[
h(k)\sim -\frac{\pi}{4k}\int_0^\infty q^2(x)dx+\frac{4\pi}{3k}\sum_j
\xi_j^3
\]
Thus, by lemma \ref{hf} from Appendix, $f(k)=h(k)$ and the lemma
easily follows.

\end{proof}\bigskip

As a simple corollary, we get

\begin{lemma}For the modified Jost function,
\begin{equation}
j_m(k)=\frac{2ikB_m(k)}{m(k^2)+ki}\exp\left( \frac{1}{2\pi i}
\int_\mathbb{R} K(t,k) \ln \left( \frac{|m(t^2)+it|^2}{4\pi
|t|\mu(t^2)} \right)dt \right), \quad k\in \mathbb{C}^+ \label{muj}
\end{equation}
\end{lemma}
This formula is the most natural one since the Blaschke product is
built not from zeroes of $j_m$, those are hiding in the poles of
$m(k^2)+ik$.

Notice also that for a.e. $k\in \mathbb{R}$,
\begin{equation}
|j_m(k)|^2=\frac{|k|}{\pi\mu(k^2)} \label{fact}
\end{equation}

\subsection{One result on weak convergence}

The main goal of this subsection  is to prove
\begin{theorem}
As $R\to\infty$,
\[
\pi |j_m(k,R)|^2d\rho(E)\to \sqrt{E}dE
\]
in weak--{\rm ($\ast$)} topology on $\mathbb{R}^+$.\label{weak}
\end{theorem}
\begin{proof}
By the Spectral Theorem,
\[
\int_\mathbb{R} \frac{u(s,k)u(t,k)}{E-z}d\rho(E)=G(s,t,z), \quad
z\in \mathbb{C}^+
\]
 where $G(s,t,z)$ is the Green function for $H$ and the identity
is understood in the distributional sense. Thus,
\[
I=\int_\mathbb{R}
\frac{Eu(s,k)u(t,k)+u'(s,k)u'(t,k)}{E-z}d\rho(E)=\delta(s-t)+zG(s,t,z)+\partial^2_{st}
G(s,t,z)
\]
Substituting the resolvent identity
\[
G=G_0-G_0\ast q\ast G
\]
and the explicit formula (\ref{reso}), we get (again, in the
distributional sense)

\begin{eqnarray*}
I=\sqrt{-z} e^{\sqrt{-z}|s-t|}-z(G_0\ast q\ast G)(s,t,z)\hspace{3cm}\\
+\left(\frac{\sqrt{-z}}{2}+\frac{1}{2\sqrt{-z}}\right)e^{\sqrt{-z}(s+t)}+(\partial_s(G_0\ast
q \ast G)\partial_t)(s,t,z)
\end{eqnarray*}
Taking imaginary part, we have for $z=x+iy$
\begin{eqnarray}
\int_\mathbb{R} \frac{y(Eu(s,k)u(t,k)+u'(s,k)u'(t,k))}{(E-x)^2+y^2}
d\rho(E)=\nonumber\\=\Im\left(\sqrt{-z} e^{\sqrt{-z}|s-t|}-z(G_0\ast
q\ast G)(s,t,z)\hspace{3cm} \right.
\nonumber\\
\left.
+\left(\frac{\sqrt{-z}}{2}+\frac{1}{2\sqrt{-z}}\right)e^{\sqrt{-z}(s+t)}+(\partial_s(G_0\ast
q \ast G)\partial_t)(s,t,z)\right) \label{id8}
\end{eqnarray}
and the integral in the l.h.s. converges for fixed $s$ and $t$ since
\[
\int \frac{d\rho(E)}{1+E^2}<\infty
\]
and $|u(s,k)|, |u(t,k)|\leq Ck^{-1}$ as $k\to+\infty$ (by
(\ref{relation}) and asymptotics of Jost functions). So, (\ref{id8})
holds as a functional identity. Take $s=t\to\infty$ and use $q\in
L^2(\mathbb{R}^+)$ to get
\begin{eqnarray}
\int_\mathbb{R} \frac{y(Eu(s,k)^2+u'(s,k)^2)}{(E-x)^2+y^2} d\rho(E)=
\Im(\sqrt{-z})+\bar{o}(1), \quad s\to\infty
\end{eqnarray}
In particular,
\begin{equation}
\sup_{s>0}\int_\mathbb{R} \frac{Eu(s,k)^2+u'(s,k)^2}{E^2+1}
d\rho(E)<\infty\label{uniform}
\end{equation}
By the standard approximation argument, (\ref{modul}), and
(\ref{uniform}) we have
\[
|j_m(k,R)|^2d\rho(E)\to \pi^{-1}\sqrt E dE
\]
in the weak--($\ast$) topology on $(0,\infty)$.
\end{proof}

\section{ The asymptotics of modified Jost function on the real line}

In this section, we will prove
\begin{theorem} For any finite interval $I$ not containing zero,

\begin{eqnarray}
\int_I \left|\frac{j_m(k,R)}{j_m(k)}-1\right|^2dk\to 0,\label{first}
\\
\int_I \left|{j_m(k,R)}\right|^2d\rho_s(E)\to 0 \label{second}
\end{eqnarray}
as $R\to\infty$.\label{asi}
\end{theorem}
\begin{proof}
The proof is standard. We will consider $I\subset \mathbb{R}^+$, the
negative half-line can be handled similarly. Fix any $a,b$ such that
$0<a<b$ and take isosceles triangle $T$ with base $[a,b]$ and the
sides $\gamma_1$, $\gamma_2$. Denote the angle between
$\gamma_{1(2)}$ and a base by
 $\phi$ and $\gamma=\gamma_1\cup\gamma_2\cup [a,b]$.
 Here, $\phi$ is a small positive
angle to be chosen later. Take some reference interior point $\xi\in
T$ and denote the harmonic measure by $\omega_\xi(t)$. We have two
simple estimates
\[
\omega_\xi(t)\sim C|t-a|^{\pi/\phi-1}, \quad t\sim a, \quad
\omega_\xi(t)\sim C|t-b|^{\pi/\phi-1}, \quad t\sim b
\]
The function $\omega_\xi(t)$ is nonnegative and smooth on the sides
of triangle.

Write
\begin{equation}
0\leq
\int_{\gamma}\left|\frac{j_m(k,R)}{j_m(k)}-1\right|^2w_\xi(k)d|k|=
\int_{\gamma}\left(\left|\frac{j_m(k,R)}{j_m(k)}\right|^2+1\right)w_\xi(k)d|k|-I
\label{aha1}
\end{equation}
and
\begin{equation}
I=2\Re \int_{\gamma}\frac{j_m(k,R)}{j_m(k)}w_\xi(k)d|k|=2\Re
\frac{j_m(\xi,R)}{j_m(\xi)}\to 2\label{aha2}
\end{equation}
Consider
\[
A=\int_{\gamma_1\cup\gamma_2}\left|\frac{j_m(k,R)}{j_m(k)}\right|^2w_\xi(k)d|k|
\]
We need the following lemma
\begin{lemma}
Consider $ g(k,R)=j_m(k,R)/j_m(k)$ and fix any $a,b>0$. Then for
$k\in \Omega=\{0<\Im k<1, a<\Re k<b$ we have
\begin{equation}
\left|g(k,R)\right|<C(\Im k)^{-1/2} \label{hardy}
\end{equation}
uniformly in $R$.
\end{lemma}
\begin{proof}The proof is elementary but lengthy.
The multiplicative representation (\ref{muj}) yields
\[
g(k,R)=I_1I_2I_3I_4
\]
where
\[
I_1=\frac{m(k^2)+ik}{m_R(k^2)+ik}\exp\left(\frac{1}{2\pi i}
\int_{a-\delta}^{b+\delta} \frac{t^2}{k^2(t-k)}\ln
\left(\left|\frac{m_R(t^2)+it}{m(t^2)+it}\right|^2\right)dt\right)
\]
\[
I_2=\exp\left(\frac{1}{2\pi i} \int_{a-\delta}^{b+\delta}
\frac{t^2}{k^2(t-k)}\ln \left({\pi
t^{-1}|j_m(t,R)|^2\mu(t^2)}\right)dt\right)
\]
\[
I_3=\exp\left(\frac{1}{2\pi i}\int_{t<a-\delta, t>b+\delta}
\frac{t^2}{k^2(t-k)} \left[\ln \left( \frac{|m_R(t^2)+it|^2}{4\pi
|t|\mu_R(t^2)}\right)-\ln \left( \frac{|m(t^2)+it|^2}{4\pi
|t|\mu(t^2)}\right)\right]dt\right)
\]
and $\delta$ is positive small.
\[
I_4=B_m(k,R)B_m^{-1}(k)
\]
$I_3$ is uniformly bounded because
\[
t^2 \ln \left( \frac{|m(t^2)+it|^2}{4\pi |t|\mu(t^2)}\right),\, t^2
\ln \left( \frac{|m_R(t^2)+it|^2}{4\pi |t|\mu_R(t^2)}\right)\in
L^1(\mathbb{R})
\]
with estimates uniform in $R$ by (\ref{relq1}).  $I_4$ is uniformly
bounded since $\|\xi_j(R)\|_{\ell^3}$ is bounded uniformly in $R$ by
(\ref{relq1}) also.\bigskip

\[
|I_2|\leq \exp\left(\Re \frac{1}{2\pi i } \int_{a-\delta}^{b+\delta}
\left(\frac{1}{(t-k)}+\frac{t+k}{k^2}\right)\ln \left({\pi
t^{-1}|j_m(t,R)|^2\mu(t^2)}\right)dt\right)
\]
We have
\begin{equation}
\int_{a-\delta}^{b+\delta}\pi t^{-1}|j_m(t,R)|^2\mu(t^2)dt<C
\label{simple}
\end{equation}
uniformly in $R$ by theorem \ref{weak} and
\[
\int_0^\infty \frac{t}{\pi|j_m(t,R)|^2 (t^4+1)}dt=\int_0^\infty
\frac{\mu_R(t^2)}{t^4+1}dt<C
\]
uniformly in $R$ by the standard estimates on the spectral measure.
Moreover,
\[
\int_{a-\delta}^{b+\delta} |\ln \mu(t^2)|dt<C
\]
by (\ref{uniformia}).

 Therefore,
\[
\int_{a-\delta}^{b+\delta} \left|\frac{t+k}{k^2}\ln \left({\pi
t^{-1}|j_m(t,R)|^2\mu(t^2)}\right)\right|dt<C
\]
 The application of
Jensen's inequality and (\ref{simple}) to
\[
\exp\left(\Re \frac{1}{2\pi i } \int_{a-\delta}^{b+\delta}
\frac{1}{t-k}\ln \left({\pi
t^{-1}|j_m(t,R)|^2\mu(t^2)}\right)dt\right)
\]
gives
\[
|I_2|<C(\Im k)^{-1/2}
\]
(the standard estimate for the boundary growth of
$H^2(\mathbb{C}^+)$ functions). We are left with estimating $I_1$.
Write
\[
I_1=\frac{P_R(k)}{P(k)}
\]
where
\[
P_R(k)=\frac{1}{m_R(k^2)+ik}\exp\left(\frac{1}{\pi i}
\int_{a-\delta}^{b+\delta} \frac{t^2}{k^2(t-k)}\ln
\left|{m_R(t^2)+it}\right|dt\right)
\]
and
\[
P(k)=\frac{1}{m(k^2)+ik}\exp\left(\frac{1}{\pi i}
\int_{a-\delta}^{b+\delta} \frac{t^2}{k^2(t-k)}\ln
\left|{m(t^2)+it}\right|dt\right)
\]
Let us show that
\[
0<C_1<|P(k)|<C_2, \quad 0<C_1<|P_R(k)|<C_2
\]
uniformly in $R$ and $k\in \Omega$. We will prove this for $P_R(k)$,
the estimates for $P(k)$ are similar. Map $k\in\{\Im k>0, \Re k>0\}$
to $z\in \mathbb{C}^+$ by $z=k^2$ and  consider
\[
Q_R(z)=\frac{1}{m_R(z)+\sqrt{-z}} \exp\left(\frac{1}{\pi i} \int_{I}
\frac{\sqrt{\xi}(\sqrt{\xi}+\sqrt{z})}{2z} \frac{1}{\xi-z}\ln
\left|{m_R(\xi)+\sqrt{-\xi}}\right|d\xi\right)
\]
where $I$ in the image of $[a-\delta,b+\delta]$. We have
\[
Q_R(z)=\frac{1}{m_R(z)+\sqrt{-z}} \exp\left(\frac{1}{\pi i} \int_{I}
 \frac{1}{\xi-z}\ln
\left|{m_R(\xi)+\sqrt{-\xi}}\right|d\xi\right)M_R(z)
\]
with
\[
M_R(z)=\exp\left(\frac{1}{\pi i} \int_{I}
\left(\frac{\sqrt{\xi}(\sqrt{\xi}+\sqrt{z})}{2z}-1\right) \ln
\left|{m_R(\xi)+\sqrt{-\xi}}\right|d\xi\right)
\]
Since
\[
0<C_1<|m_R(\xi)+\sqrt{-\xi}|<C_2|m_R(\xi)|+C_3
\]
and
\[
\int_I |m_R(\xi)|^pd\xi<C(p), \quad p<1
\]
uniformly in $R$ (by Kolmogorov's theorem), we have
\[
0<C_1<|M_R(z)|<C_2
\]
uniformly in $R$ and  $z\in \{\Omega'=\Omega^2\}$. The uniform in
$R$ estimate
\[
0<C_1<\left|\frac{1}{m_R(z)+\sqrt{-z}} \exp\left(\frac{1}{\pi i}
\int_{I}
 \frac{1}{\xi-z}\ln
\left|{m_R(\xi)+\sqrt{-\xi}}\right|dt\right)\right|<C_2
\]
follows from lemma \ref{l6} in  Appendix.
\end{proof}
By (\ref{hardy}), we have

\begin{equation}
A\to \int_{\gamma_1\cup\gamma_2}w_\xi(k)d|k|, \quad
R\to\infty\label{aha3}
\end{equation}
provided that $\phi$ is taken small.
\[
\int_{[a,b]}\left|\frac{j_m(k,R)}{j_m(k)}\right|^2w_\xi(k)d|k|+\pi/2
\int_{[a^2,b^2]}\left|{j_m(k,R)}\right|^2w_\xi(k)k^{-2}d\rho_s(E)=
\]
\begin{equation}
=\pi/2 \int_{[a^2,b^2]}\left|{j_m(k,R)}\right|^2w_\xi(k)
k^{-2}d\rho(E)\to \int_{[a,b]}w_\xi(k)dk\label{aha4}
\end{equation}
by theorem \ref{weak}. Combining  (\ref{aha1}), (\ref{aha2}),
(\ref{aha3}), and (\ref{aha4}), we get the statement of the theorem.
\end{proof}
\section{Asymptotics of solutions to wave equation}

 By the Spectral Theorem, we have the following
representation for the solution to (\ref{wave}) with Cauchy data
 orthogonal to eigenstates corresponding to negative
eigenvalues:

\begin{equation}
y(x,t)=\int_0^\infty
\cos(kt)u(x,k)\breve{\phi}(k)d\rho(E)+\int_0^\infty
\frac{\sin(kt)}{k} u(x,k)\breve{\psi}(k)d\rho(E)\label{spth}
\end{equation}
where $\breve{\phi}, \breve{\psi}$ are generalized Fourier
transforms of $\phi$ and $\psi$ associated with $H$. If one wants
(\ref{spth}) to give the classical solution for $t>0$, it is
sufficient to assume that, $\phi,\psi \in \cal{D}(H^\alpha)$ for
large $\alpha$ to guarantee the absolute convergence of the
integrals after differentiation (e.g.,  $\alpha\geq 3/2$ is
sufficient). We need to assume that
$P_{(-\infty,0)}\phi=P_{(-\infty,0)}\psi=0$ since otherwise the
solutions can grow exponentially and thus are not physical.

 Consider
the second integral in (\ref{spth}). We assume
\begin{equation}
\breve{\psi}(k)/k\in L^2 (d\rho(k)) \label{asum1}
\end{equation}
which is equivalent to $\psi=\sqrt{|H|}g$ for some $g$. This
assumption is rather natural. Indeed, if $H$ has eigenvalue at zero
with eigenfunction $\phi$, then its contribution will yield a linear
growth
\[
\|u(x,t)\|_2\geq Ct
\]
This will occur even in the free case when $\psi$ does not
oscillate.  Although some analysis is possible in the general case,
we will work with stable situation when the $L^2(\mathbb{R})$ norm
of the solution is bounded in time and so will require
(\ref{asum1}). We prove the pointwise ``large time -- large
coordinate" translation--like asymptotics for the initial data in
the set dense in $\sigma_{\rm{ac}}(H)$.  For large $t$ and small
$x$, the solution can behave in rather chaotic way due to the
possible presence of embedded singular continuous spectrum.

\begin{theorem}
Assume that $q\in L^2(\mathbb{R}^+)$ and
\[
\int_0^\infty q(x)dx
\]
converges. Take $f\in \cal{{D}}(\sqrt{|H|}), P_{(-\infty,0)}f=0$ and
let $\phi=f$ and $\psi=-i\sqrt{|H|}f$. Then, for any fixed $x$, we
have
\[
y(T+x,T)\to \mu_f(x), \quad T\to\infty
\]
where $\mu_f(x)\in W^{1,2}(\mathbb{R})$.\label{ch1}
\end{theorem}
\begin{proof}
Clearly, for this choice of initial conditions,
\[
y(x,t)=e^{-it\sqrt{|H|}}f
\]
and we need to control $e^{-iT\sqrt{|H|}}f$ suitably scaled. If $H$
has zero eigenvalue with eigenfunction $e_0(x)$, then the solution
of wave equation with initial condition $\{e_0(x),0\}$ is
stationary, i.e. $y_0(x,t)=e_0(x)\to 0$ as $x\to\infty$ so we can
always assume without loss of generality that $f$ is orthogonal to
$e_0$. We have
\begin{equation}
\int (k^2+1)|\breve{f}(k)|^2d\rho(E)<\infty \label{damp}
\end{equation}
which guarantees the absolute convergence of the integrals in
(\ref{spth}) due to (\ref{uniform}).

Take any $\delta, M$, $0<\delta<M$. Then
\[
y(T+x,T)=I_1+I_2+I_3+I_4
\]
where
\[
I_1=\int_\delta^M e^{-ikT}\breve{f}(k)u(T+x,k)d\rho_s(E)
\]

\begin{eqnarray*}
I_2= \int_{\delta}^M e^{-ikT} \breve{f}(k) u(T+x,k)\mu(k^2)dE
\end{eqnarray*}

\[
I_3=\int_0^\delta e^{-ikT}\breve{f}(k)u(T+x,k)d\rho(E)
\]
\[
I_4=\int_M^\infty e^{-ikT}\breve{f}(k)u(T+x,k)d\rho(E)
\]
By (\ref{relation}), (\ref{second}), and (\ref{damp}) $I_1$
converges to zero for $T\to\infty$.

Consider $I_3$. Let
\[
g(x)=\int_0^\delta e^{-ikT}\breve{f}(k)u(x,k)d\rho(E)
\]
By the Spectral Theorem, we have $\|g\|_2\to 0, \|g'\|_2\to 0$ as
$\delta\to0$ and so $\|g\|_\infty\to 0$ as $\delta\to 0$.

For $I_4$, Cauchy-Schwarz yields
\[
\sup_{T,x}|I_4|\leq \sup_{T,x}\left(\int_M^\infty
|\breve{f}(k)|^2(k^2+1)d\rho(E)\right)^{1/2}\left(\int_M^\infty
\frac{u(T+x,k)^2}{E+1}d\rho(E)\right)^{1/2}\to 0 ,\quad
\]
as $M\to\infty$ by (\ref{uniform}) and (\ref{damp}).

 For $I_2$, we have ($\delta_1=\sqrt\delta, M_1=\sqrt{M}$)
\begin{eqnarray*}
I_2=\frac{1}{\pi i} \int_{\delta_1}^{M_1} e^{-ikT}
\breve{f}(k)\left(\overline{j_m(T+x,k)}e^{ik(T+x)-i\phi(0,k,T+x)}\right.\\
\left.-j_m(T+x,k)e^{-ik(T+x)+i\phi(0,k,T+x)}\right)\frac{kdk}{|j_m(k)|^2}
\end{eqnarray*}
where we used (\ref{fact}) and (\ref{relation}). By (\ref{damp}),
\[
\breve{f}(k)=j_m(k)h(k)k^{-1}
\]
and $(k^2+1)^{1/2}h(k)\in L^2(\mathbb{R}^+)$.

For any $\delta_1,M_1>0$,
\begin{eqnarray*}
\int_{\delta_1}^{M_1} e^{-ikT} {h(k)}
\left(\frac{\overline{j_m(T+x,k)}}{\overline{j_m(k)}}\,e^{ik(T+x)-i\phi(0,k,T+x)}\right.\\
\left.-\frac{{j_m(T+x,k)}}{{j_m(k)}}\frac{j_m(k)}{\overline{j_m(k)}}\,e^{-ik(T+x)+i\phi(0,k,T+x)}\right)dk
\end{eqnarray*}
converges to
\[
\exp\left(-\frac{i}{2k}\int_0^\infty
q(x)dx\right)\int_{\delta_1}^{M_1} h(k)e^{ikx}dk
\]
by (\ref{first}) and Riemann-Lebesgue lemma. Since $M_1$ and
$\delta_1$ are arbitrary and $h$ decays fast, the theorem is true
with
\[
\mu_f(x)=\frac{1}{\pi i}\left(\int_0^\infty
h(k)e^{ikx}dk\right)\exp\left(-\frac{i}{2k}\int_0^\infty
q(x)dx\right)
\]
\end{proof}
Taking the conjugation, one can prove similar result for initial
conditions of the form $\{f,i\sqrt{|H|}f\}$. Since
\[
\{f,\sqrt{|H|}g\}=\{f_1,i\sqrt{|H|}f_1\}+ \{f_2,-i\sqrt{|H|}f_2\}
\]
with $f_1=(f-ig)/2$ $f_2=(f+ig)/2$, the theorem holds in general
case. Clearly, the simplest way to satisfy $\psi=\sqrt{|H|}g$ with
some $g$ is to take $\psi=0$.

\begin{remark}
How should one take a function $f$ to guarantee that at least part
of the wave will travel ballistically after taking $P_{[0,\infty)}f$
as initial value in Cauchy problem? The theorem just proved tells us
that $f$ should have some nontrivial part in a.c. subspace. Checking
that is not easy in general but if one takes $f$ to be compactly
supported, then $d(P_Ef,f)/dE>0$ for a.e. $E>0$.  The latter
statement can be easily proved by adjusting technique from
\cite{den8}. The theorem \ref{ch1} requires the convergence of the
integral
\begin{equation}
\int_0^\infty q(x)dx \label{ahah}
\end{equation}
That can not be achieved for $q\notin L^1(\mathbb{R^+})$ unless $q$
oscillates but if it does, then the negative spectrum might easily
appear and therefore we have to project away from the negative
eigenspace which is not an easy thing to do (though possible of
course by Gram-Schmidt process). In the next section, we will not
require (\ref{ahah}) and so the asymptotics will be established for
more tangible set of initial data.
\end{remark}
\section{Modified wave operators}

In the previous section, we assumed the conditional convergence to
prove that the part of initial value that corresponds to a.c.
spectrum gives that portion of a wave that undergoes the simple
translation. If the potential is only square summable, this is not
the case anymore. We will study the long-time dynamics by
considering the modified wave operators. Notice that if $H\geq 0$,
then the group $e^{it\sqrt{H}}$ gives the formal solution to the
wave equation.

Let us start with some definition. Take any function $f(x)\in
L^2(\mathbb{R^+})$. Then
\begin{eqnarray*}
e^{it\sqrt{H_0}}f= \frac{2}{\pi}\int_0^\infty e^{ikt}\sin
(kx)\tilde{f}(k)dk= \hspace{3cm}\\
=\frac{1}{2\pi}\left( \int_0^\infty
e^{ik(t-x)}\hat{f}_o(k)dk- \int_0^\infty
e^{ik(t+x)}\hat{f}_o(k)dk\right)
\end{eqnarray*}
where
\[
\tilde{f}(k)=\int_0^\infty f(x)\sin (kx)dx=-i\hat{f_o}(k)/2
\]
$f_o(x)$ is the odd continuation of $f(x)$ and $\hat{f_o}(k)$ is its
(inverse) Fourier transform. Thus, for $t\to+\infty$
\[
e^{it\sqrt{H_0}}f\sim \frac{1}{2\pi} \int_0^\infty
e^{ik(t-x)}\hat{f}_o(k)dk
\]
in $L^2(\mathbb{R}^+)$. Consider the multiplier
\[
M(t,k)=\exp\left(ikt+\frac{i}{2k}\int_0^t q(s)ds\right)
\]
and
\[
[W(t)f](x)=\frac{1}{2\pi} \int_0^\infty e^{-ikx}
M(k,t)\hat{f}_o(k)dk
 \]

The main result of this section is
\begin{theorem}
Let $q\in L^2(\mathbb{R}^+)$ and
\begin{equation}
|q(x)|<C(1+x)^{-1/2} \label{ops}
\end{equation}
If $H_1=P_{[0,\infty)}H$, then the following strong limit exists
\[
{\rm (s)}-\lim_{t\to+\infty} e^{-it\sqrt{H_1}}W(t)
\]
in $L^2(\mathbb{R}^+)$ norm. \label{mmm}
\end{theorem}

\begin{proof}
Since $W(t)$ is unitary, it is sufficient to prove existence of the
limit for functions $f(x)$ with $\hat{f}_o(k)$ being infinitely
smooth and with compact support on, say, $[a,b]\subset
\mathbb{R}^+$. Let
\[
\omega(t)=\left|\int_0^t |q(x)|dx/\sqrt{t}\right|
\]
Obviously, $\omega(t)\to 0$ as $t\to\infty$ but
$\omega(t)>Ct^{-1/2}$ unless $q=0$.
\begin{lemma} Consider
\[
g_t(x)=\int_0^\infty
\exp\left(-ikx+\frac{i}{2k}\int_0^tq(s)ds\right)\hat{f}_o(k)dk
\]
Then,
\begin{equation}
\int\limits_{|x|>\sqrt{t\omega(t)}} |g_t(x)|^2dx\to 0, \quad
t\to\infty\label{loca}
\end{equation}\label{pyat1}
\end{lemma}
\begin{proof}
We have
\[
\int x^2 |g_t(x)|^2dx\lesssim \int
|\hat{f}'_o(k)|^2dk+\left(\int_0^t |q(s)|ds\right)^2 \int
k^{-2}|\hat{f}_o(k)|^2 dk\lesssim 1+t\omega^2(t)
\]
which proves the claim.
\end{proof}
Lemma says that $W(t)f$ is localized on $|x-t|<\sqrt{t\omega(t)}$ in
$L^2$ sense. We also need the following lemma
\begin{lemma} We have
\[
\sup_{\alpha,\beta,k,t}\left|\int_{\alpha}^\beta
e^{ixk}[W(t)f](x)dx\right|<\infty
\]
\label{sing}
\end{lemma}
\begin{proof}
It is sufficient to prove
\[
\sup_{\alpha,\beta,k, T_1}\left|\int_0^\infty
\frac{e^{i(\xi-k)\beta}-e^{i(\xi-k)\alpha}}{\xi-k}\exp\left(\frac{iT_1}{\xi}\right)\hat{f}_o(\xi)d\xi\right|<\infty
\]
Since $\hat{f}_o$ is infinitely smooth with support on $[a,b]\subset
(0,\infty)$, we just need to show that, say,
\[
\sup_{\gamma,T_1}\left|\int_{1/2}^2
\frac{e^{i(\xi-1)\gamma}}{\xi-1}\exp\left(\frac{iT_1}{\xi}\right)d\xi\right|<\infty
\]
where integral is understood in v.p. sense. This is proved in lemma
\ref{osc} in Appendix.
\end{proof}

Denote the generalized Fourier transform of
$d_t(x)=\chi_{|x-t|<\sqrt{t\omega(t)}}W(t)f$ by $\breve{d}_t(k)$. We
have
\[
\breve{d}_t(k)=\frac{1}{2ik}\int_{t-\sqrt{t\omega(t)}}^{t+\sqrt{t\omega(t)}}
\Bigl(\overline{j_m(k,x)}e^{ixk-i\phi(0,k,x)}-j_m(k,x)e^{-ixk+i\phi(0,k,x)}\Bigr)[W(t)f]
(x)dx
\]
\[
=I_1-I_2
\]
where
\[
I_1=\frac{1}{2ik}\int_{t-\sqrt{t\omega(t)}}^{t+\sqrt{t\omega(t)}}
\overline{j_m(k,x)}e^{ixk-i\phi(0,k,x)}[W(t)f](x)dx
\]
Integrating by parts,
\begin{equation}
I_1=\frac{1}{2ik}\overline{j_m}(k,t-\sqrt{t\omega(t)})\exp(-i\phi(0,k,t-\sqrt{t\omega(t)}))
\int_{t-\sqrt{t\omega(t)}}^{t+\sqrt{t\omega(t)}}
e^{ixk}[W(t)f](x)dx+J_1\label{idio1}
\end{equation}
\[
J_1=\frac{1}{2ik}\int_{t-\sqrt{t\omega(t)}}^{t+\sqrt{t\omega(t)}}
(\overline{j_m(k,x)}\exp(-i\phi(0,k,x)))_x'\left(\int_x^{t+\sqrt{t\omega(t)}}e^{iku}
[W(t)f](u)du\right)dx
\]
By (\ref{derivative}) and the lemma \ref{sing},
\[
|J_1|\lesssim\int_{t-\sqrt{t\omega(t)}}^{t+\sqrt{t\omega(t)}}
|q(x)||j_m(k,x)|dx
\]
for any $k\in I\subset \mathbb{R}^+$. So, by Minkowski and
(\ref{uniform})
\begin{equation}
\left(\int_I
|J_1|^2d\rho(E)\right)^{1/2}\lesssim\int_{t-\sqrt{t\omega(t)}}^{t+\sqrt{t\omega(t)}}
|q(x)|\left(\int_I |j_m(k,x)|^2 d\rho(E)\right)^{1/2}dx \label{oups}
\end{equation}
\[
<C\sqrt{\omega(t)}\to 0
\]
The first term in (\ref{idio1}) can be written as
\[
K_1=\frac{1}{2ik}\overline{j_m}(k,t-\sqrt{t\omega(t)})
\hat{f}_o(k)\exp\left(ikt+\frac{i}{2k}\int_{t-\sqrt{t\omega(t)}}^t
q(u)du\right)
\]
\[
+\frac{1}{2ik}\overline{j_m}(k,t-\sqrt{t\omega(t)})\exp(-i\phi(0,k,t-\sqrt{t\omega(t)}))
\int_{|x-t|>\sqrt{t\omega(t)}} e^{ixk}[W(t)f](x)dx
\]
The second term converges to zero in $L^2(I,d\rho_s(E))$ because
\[
\int_{|x-t|>\sqrt{t\omega(t)}} e^{ixk}[W(t)f](x)dx
\]
is uniformly bounded  by lemma \ref{sing} and
\[
\int_I |j_m(k,t)|^2d\rho_s(E)\to 0, \quad t\to\infty
\]
by (\ref{second}). It also converges to zero in $L^2(I,\mu(k^2)dE)$
by (\ref{first}) and since
\[
\left\|\int_{|x-t|>\sqrt{t\omega(t)}}
e^{ixk}[W(t)f](x)dx\right\|_{L^2(dk,\mathbb{R})}\to 0, \quad
t\to\infty
\]
by (\ref{loca}). Therefore, by (\ref{first}) again

\[
e^{-ikt}K_1\to\frac{1}{2ik}\overline{j_m}(k)\hat{f}_o(k)\chi_\Theta(E),
\quad t\to\infty
\]
in $L^2(I,d\rho(E))$ where $\Theta$ is the complement of the support
of $d\rho_s(E)$ and $I\subset(0,\infty)$. The estimates for the
$I_2$ are similar, they yield
\[
\int_I |I_2|^2d\rho(E)\to 0, t\to\infty
\]

Thus, we have
\[
\int_I
\left|\breve{d}_t(k)e^{-ikt}-\frac{1}{2ik}\overline{j_m}(k)\hat{f}_o(k)
\chi_\Theta(E)\right|^2d\rho(E)\to 0,\quad t\to\infty
\]
for any $I\subset \mathbb{R}^+$. If the generalized Fourier
transform of $w_t(x)=W(t)f$ is denoted by $\breve{w}_t(k)$, then
\[
\int_I
\left|\breve{w}_t(k)e^{-ikt}-\frac{1}{2ik}\overline{j_m}(k)\hat{f}_o(k)
\chi_\Theta(E)\right|^2d\rho(E)\to 0,\quad t\to\infty
\]
for any interval $I=[\delta,M]$. On the other hand, notice that
$w_t(x)\in W^{l,2}(\mathbb{R})$ for any $l\geq 0$,
$\sup_{t>0}\|\partial^2_{xx}w_t(x)\|_2<\infty$, and $w_t(x)$ is
concentrated around $t$ by lemma~\ref{pyat1}. If $p(x)$ is
infinitely smooth, $p(x)=0$ for $x<0$ and $p(x)=1$ for $x>1$, then
$p(x)w_t(x)\in \cal{D}(H)$ and $\sup_{t}\|H(p(x)w_t(x))\|_2<\infty$.
Therefore,
\[
\limsup_{t\to\infty}\int_M^\infty \|\breve{w}_t(k)\|_2^2d\rho(E)\to
0, \quad M\to\infty
\]
and
\[
\int_\delta^\infty
\left|\breve{w}_t(k)e^{-ikt}-\frac{1}{2ik}\overline{j_m}(k)\hat{f}_o(k)
\chi_\Theta(E)\right|^2d\rho(E)\to 0
\]
for any $\delta>0$. \bigskip

Let us show that no $L^2(d\rho(E))$ norm of $\breve{w}_t(k)$ can
concentrate around zero energy. In other words, that
\begin{equation}
\limsup_{t\to\infty} \|P_{[0,\delta]}W(t)f\|\to 0\label{trick}
\end{equation}
as $\delta\to 0$.

By definition of $W(t)$, we have
\[
W(t)f=\upsilon_t''(x-t)
\]
where $\upsilon_t \in W^{2,2}(\mathbb{R})$ and its $L^2$ norm is
concentrated on $[-\sqrt{t},\sqrt{t}]$. Thus
\[
W(t)f=Hs_1(t)+s_2(t)
\]
where $s_1(t)\in \cal{D}(H)$ and
\[
\sup_{t>0}\Bigl(\|s_1(t)\|_2+\|Hs_1(t)\|_{2}\Bigr)<\infty
\]
For $s_2$,
\[
\|s_2\|_{L^2(\mathbb{R}^+)}\to 0, \quad t\to\infty
\]
Therefore,
\[
\int_0^\delta |\breve{w}_t(k)|^2d\rho(E)\lesssim \int_0^\delta
\left(E|\breve{s}_1(k)|^2+|\breve{s}_2(k)|^2\right)d\rho(E)\lesssim
\delta+\bar{o}(1), \quad t\to\infty
\]
Therefore, we have (\ref{trick}) and so
\[
\int_0^\infty
\left|\breve{w}_t(k)e^{-ikt}-\frac{1}{2ik}\overline{j_m}(k)\hat{f}_o(k)
\chi_\Theta(E)\right|^2d\rho(E)\to 0
\]
which proves the theorem.
\end{proof}\bigskip

If the potential is nonnegative then $H\geq 0$ and we do not have to
project to the positive part of the spectrum. We were working with
the group $e^{it\sqrt{H_1}}$ to guarantee stability in
$L^2(\mathbb{R})$, after all we established asymptotics in this
norm.
\bigskip

Using simple contour integration technique, one can show that

\[
\sup_{|x|<T}\int_0^\infty
\exp\left(-ikx+\frac{iT_1}{k}\right)\hat{f}_o(k)dk\to 0
\]
as $T_1\to\infty$ provided that infinitely smooth function $f_o$ has
compact support in $(0,\infty)$ (in fact, the bound is $\sim
T_1^{-1/2}$). That shows that the $L^2$ norm of $W(t)f$ might be
smeared over the interval $[t-\sqrt{t},t+\sqrt{t}]$ and the strength
of smearing depends on the size of
\[
\int_0^T q(x)dx
\]
This required us to make an additional assumption (\ref{ops}). In
the meantime, some results can be proved even without (\ref{ops}).
For example, the methods presented in \cite{den9}, allow us to prove
\begin{proposition}
Let  $q\in L^2(\mathbb{R}^+)$. If $H_1=P_{[0,\infty)}H$, then for
any $f\in L^2(\mathbb{R}^+)$
\[
\frac{1}{T}\int_0^T
\left\|e^{-it\sqrt{H_1}}W(t)f-G(f)\right\|_2^2dt\to 0
\]
\end{proposition}
The formula for $G(f)$ can be given explicitly in terms of
generalized Fourier transform. $\Bigl($ We had to use (\ref{ops}) at
(\ref{oups}). In the meantime, we have
\begin{eqnarray*}
\frac{1}{T}\int_0^T \left(\int_{t-\sqrt{t}}^{t+\sqrt{t}}
|q(x)|dx\right)^2dt\lesssim \frac{1}{{T}} \int_0^T
\sqrt{t}\int_{t-\sqrt{t}}^{t+\sqrt{t}} q^2(x)dxdt\\
\lesssim \frac{1}{T} \int_0^{CT} (1+x)q^2(x)\to 0, \quad T\to\infty
\end{eqnarray*}
assuming only $q\in L^2(\mathbb{R}^+)$. That leads to the proof of
proposition.$\Bigr)$
\bigskip

The existence of wave operators in the form we proved also
establishes the long-time asymptotics for the initial data from the
a.c. part of the spectrum. If $f$, the initial value for
$e^{it\sqrt{H}}$ is compactly supported, then its spectral measure
has a.c. part supported on $\mathbb{R}^+$. If $f_1=P_{\rm ac}f,
f_2=P_{\rm s}f$, and $y_1(t)=\exp(it\sqrt{H})f_1,
y_2(t)=\exp(it\sqrt{H})f_2$, then $y_1\perp y_2$ for any $t$. We do
not know whether $y_2$ travels ballistically in $L^2$ norm, but
since $y_1$ does (for most time if $q\in L^2(\mathbb{R}^+)$ and for
all times if we additionally assume (\ref{ops})), the whole wave
$y=y_1+y_2$ must travel ballistically too. Indeed,
\begin{eqnarray*}
\int_{t-\sqrt{t}}^{t+\sqrt{t}} |y(x,t)|^2dx=
\int_{t-\sqrt{t}}^{t+\sqrt{t}}
|y_1(x,t)|^2dx+\int_{t-\sqrt{t}}^{t+\sqrt{t}} |y_2(x,t)|^2dx\\
+2\Re \int_{t-\sqrt{t}}^{t+\sqrt{t}} y_1(x,t)\overline{y_2(x,t)}dx
\end{eqnarray*}
and since
\[
\int_{t-\sqrt{t}}^{t+\sqrt{t}} y_1(x,t)\overline{y_2(x,t)}dx\to 0
\]
by orthogonality and localization of $y_1$, we have that
\[
\limsup_{t\to\infty}\int_{t-\sqrt{t}}^{t+\sqrt{t}} |y(x,t)|^2dx\geq
\|f_1\|_2^2
\]
if $q$ satisfies conditions of theorem \ref{mmm}. If $q$ is only
square summable, we still have ballistic propagation for some part
of the wave for most of the time. Notice also that the speed of
propagation is always finite, i.e. if the initial data is compactly
supported on $(0,a)$ then $y(x,t)=0$ for $x>a+t$. Thus the motion
can not be faster then ballistic. The finite speed of propagation
follows, e.g., from the Duhamel expansion.\bigskip

As was mentioned earlier, for any $p>2$ there is some $q\in
L^p(\mathbb{R}^+)$ such that the spectrum is  pure point. This, of
course, leads to localization of the wave. Indeed, for any $f$,
\[
P_{[0,\infty)}f=\sum_{j} \langle f,e_j\rangle e_j
\]
where $e_j$ are $L^2(\mathbb{R}^+)$--normalized eigenfunctions
corresponding to nonnegative eigenvalues $\lambda_j$. Therefore,
\[
e^{it\sqrt{H_1}}P_{[0,\infty)}f=\sum_{j=1}^N e^{i\sqrt{\lambda_j}t}
\langle f,e_j\rangle e_j+\epsilon_N(t)
\]
where $\sup_t\|\epsilon_N(t)\|_2\to 0$ as $N\to \infty$. The first
term, though, is localized around the origin in $L^2(\mathbb{R}^+)$
for any time so we have localization of the whole wave.
\section{Appendix}

In this appendix, we collected some simple results we used in the
main text.

\begin{lemma}\label{hf}
Assume that $g(k)$ is positive harmonic function in $\mathbb{C}^+$
and
\[
yg(iy)\to 0
\]
as $y\to+\infty$. Then, $g(k)=0$
\end{lemma}
\begin{proof}
We have integral representation for $g$
\begin{equation}
g(iy)=\beta y+\int_\mathbb{R}\frac{y}{y^2+t^2}d\tau(t)\label{ir}
\end{equation}
where $\beta\geq 0$ and for the positive measure $d\tau$
\[
\int_\mathbb{R} \frac{d\tau(t)}{1+t^2}<\infty
\]
Multiply the both sides of (\ref{ir}) by $y$ and take $y\to\infty$.
Then, $\beta=0$ and $d\tau=0$  so we get the statement of the lemma.
\end{proof}

\begin{lemma}\label{l6}
 Assume that $f(z)$ is Herglotz function in $\mathbb{C}^+$.
Take any interval $[a,b]$ and $\delta>0$. Then,
\begin{equation}
0<C_1<\left|\frac{1}{f(z)}\exp\left(\frac{1}{\pi
i}\int_{a-\delta}^{b+\delta}\frac{1}{\xi-z}\ln
|f(\xi)|d\xi\right)\right|<C_2\label{ini1}
\end{equation}
uniformly in $\Omega_1=\{0<\Im z<1, a<\Re z<b\}$. The constants
$C_{1(2)}$ depend only on $|f(i)|$.
\end{lemma}
\begin{proof} Map $\mathbb{C}^+$ to $\mathbb{D}$ by
$k(z)=(z-i)(z+i)^{-1}$. If $z(k)=i(k+1)(1-k)^{-1}$ and
$F(k)=f(z(k))$, then (\ref{ini1}) is equivalent to proving
\[
0<C_1<\left|\frac{1}{F(k)} \exp\left(\frac{1}{\pi i}
\int_{\theta_1-\delta_1}^{\theta_2+\delta_2} \frac{1}{z(u)-z(k)}\ln
|F(u)|dz(u) \right)\right|<C_2
\]
in the subset $\Omega_2=k(\Omega_1)$ of $\mathbb{D}$ adjacent to the
arc $k([a,b])=(e^{i\theta_1}, e^{i\theta_2})$. Since $F$ is outer,
we have
\[
F(k)=\exp\left(\frac{1}{2\pi}\int_{-\pi}^\pi
\frac{e^{i\theta}+k}{e^{i\theta}-k} \ln
|F(e^{i\theta})|d\theta\right)
\]
After substitution, it is sufficient to show that
\begin{equation}
\left|\int_{\theta_1-\delta_1}^{\theta_2+\delta_1}
\left(\frac{(1-k)e^{i\theta}}{\pi(1-e^{i\theta})(e^{i\theta}-k)}-
\frac{e^{i\theta}+k}{2\pi(e^{i\theta}-k)}\right)\ln
|F(e^{i\theta})|d\theta\right|<C \label{e1}
\end{equation}
\begin{equation}
\left|\int_{\mathbb{T}\backslash
[\theta_1-\delta_1,\theta_2+\delta_1]}\frac{e^{i\theta}+k}{e^{i\theta}-k}
\ln |F(e^{i\theta})|d\theta\right|<C\label{e2}
\end{equation}
Notice that we have
\[
\int_\mathbb{T} |F(e^{i\theta})|^pd\theta<C(p,|F(0)|), \quad p<1
\]
by Kolmogorov's theorem and $\ln|F(e^{i\theta})|\in
L^1(\mathbb{T})$. These estimates prove (\ref{e1}) and (\ref{e2}) in
$\Omega_2$.
\end{proof}
\begin{lemma}\label{osc}
The following estimate holds
\[
\sup_{\gamma,T_1}\left|\int_{1/2}^2
\frac{e^{i(\xi-1)\gamma}}{\xi-1}\exp\left(\frac{iT_1}{\xi}\right)d\xi\right|<\infty
\]
where integral is understood in v.p. sense.
\end{lemma}

\begin{proof} Interesting cases are when $|T_1|$ or $|\gamma|$ are large.
Making the substitution $\gamma(\xi-1)=\xi_1$ and absorbing the
constants, we need to show
\[
\sup_{\gamma, T_2}\left|\int_{-\gamma/2}^{\gamma}
\frac{e^{i\xi_1}}{\xi_1}\exp\left(\frac{iT_2}{\xi_1+\gamma}\right)d\xi_1\right|<C
\]
or, essentially,
\[
\sup_{\gamma, T_2}\left|\int_{-\gamma/2}^{\gamma/2}
\frac{e^{i\xi_1}}{\xi_1}\exp\left(\frac{iT_2}{\xi_1+\gamma}\right)d\xi_1\right|<C
\]
Let $\xi_2=\xi_1+\gamma$,
then
\[
\sup_{\gamma, T_2}\left|\int_{\gamma/2}^{3\gamma/2}
\frac{\exp\left({i(\xi_2+\frac{T_2}{\xi_2}})\right)}{\xi_2-\gamma}d\xi_2\right|=
\sup_{a, T}\left|\int_{a}^{3a}
\frac{\exp\left({iT(\xi\pm {\xi}^{-1}})\right)}{\xi-2a}d\xi\right|
\]
We can always assume that $a,T>0$. Consider the first case, i.e.

\[
\sup_{a, T}\left|\int_{a}^{3a}
\frac{\exp\left({iT(\xi- {\xi}^{-1}})\right)}{\xi-2a}d\xi\right|
\]

The boundedness of the last expression can be proved by contour integration.
Indeed, take the contour $\Gamma_{a,\delta}=\gamma_1\cup\gamma_2\cup\gamma_3$, where
\[
\gamma_1=\{\delta<|\xi-2a|<a, \quad \Im\xi=0\}
\]
\[
\gamma_2=\{\Im \xi\geq 0, |\xi-2a|=a
\]
\[
\gamma_3=\{\Im \xi\geq 0, |\xi-2a|=\delta
\]
Notice that
 $\Im \left(\xi-\xi^{-1}\right)\geq 0$ on $\Gamma_{a,\delta}$. Therefore,
\[
\left|\int_{a}^{3a} \frac{\exp\left({iT(\xi-
{\xi}^{-1}})\right)}{\xi-2a}d\xi\right|=\lim_{\delta\to
0}\left|\int_{\gamma_2\cup\gamma_3}
\frac{\exp\left({iT(\xi-{\xi}^{-1}})\right)}{\xi-2a}d\xi\right|<C
\]
with constant $C$ independent of $a$ and $T$. Consider the second
case,
\[
\sup_{a, T}\left|\int_{a}^{3a}
\frac{\exp\left({iT(\xi+ {\xi}^{-1}})\right)}{\xi-2a}d\xi\right|
\]
If $\Im \xi, \Re \xi>0$ and $|\xi|>1$, then the same contour integration gives boundedness. Thus,
\[
\sup_{a>1, T>0}\left|\int_{a}^{3a}
\frac{\exp\left({iT(\xi+ {\xi}^{-1}})\right)}{\xi-2a}d\xi\right|<C
\]
If $a<1/3$, we can again use contour integration but in $\{\Im
\xi<0, \Re \xi>0\}$. The contour should again be taken as the union
of semicircles. We are left with $a\sim 1/2$ case. If $a>1/2$, then,
again, use contour integration where $\gamma_{1(3)}$ are the same
but $\gamma_2$ is the union of arcs on the circles $|\xi-2a|=a$ and
$|\xi|=1$. Then,

\[
\sup_{1>a>1/2, T>0}\left|\int_{a}^{3a} \frac{\exp\left({iT(\xi+
{\xi}^{-1}})\right)}{\xi-2a}d\xi\right|\lesssim 1 +\sup_{1>a>1/2,
T>0}\left|\int_{-\theta_1}^{\theta_1}\frac{e^{iT\cos\theta}}{2a-e^{i\theta}}e^{i\theta}d\theta\right|
\]
where $\theta_1$ depends on $a$. The Taylor expansion around zero
gives
\[
\sup_{1>a>1/2,
T>0}\left|\int_{-\theta_1}^{\theta_1}\frac{e^{iT\cos\theta}}{2a-e^{i\theta}}
e^{i\theta}d\theta\right| \lesssim1+ \sup_{a>1/2,
T>0}\left|\int_{-\theta_1}^{\theta_1}\frac{e^{iT\cos\theta}}{2a-1-{i\theta}}d\theta\right|
\]
\[
\lesssim 1+\int_{-\theta_1}^{\theta_1}
\frac{2a-1}{(2a-1)^2+\theta^2}d\theta<C
\]
by canceling the odd part. The case $a<1/2$ is similar and the
calculation for $a=1/2$ follows by similar contour integration and
from the bound
\[
\sup_{\delta>0, T>0}\left|\int_{\delta<|\theta|<\theta_1}\frac{e^{iT\cos\theta}}{1-e^{i\theta}}e^{i\theta}d\theta\right|<C
\]
\end{proof}
\bigskip\bigskip{\bf Acknowledgements.} This research was supported by Alfred
P. Sloan Research Fellowship and NSF Grant DMS-0758239.

\end{document}